\documentclass[a4paper, english, final]{scrartcl}

\title{Diagonal Couplings of\\Quantum Markov Chains}
\author{Burkhard K\"ummerer and Kay Schwieger}
\date{}

\KOMAoptions{parskip=half, abstract=on}

\usepackage{fixltx2e}
\usepackage[T1]{fontenc}
\usepackage{lmodern}
\usepackage{microtype}			
\usepackage[english]{babel}
\usepackage{amsmath, amsthm, amssymb, bbm, mathtools, nicefrac}
\usepackage{enumitem}
	\newlist{equivalence}{enumerate}{1}
	\setlist[equivalence]{label=(\alph*)}
\usepackage{varioref, hyperref}
	\labelformat{equation}{(#1)} \let\eqref\ref
\usepackage{tikz}
	\usetikzlibrary{matrix,arrows,shapes}
\usepackage{xspace}

\theoremstyle{plain}
	\newtheorem{thm}{Theorem}[section]
	\newtheorem{prop}[thm]{Proposition}
	\newtheorem{lemma}[thm]{Lemma}
	
\theoremstyle{definition}
	\newtheorem{defn}[thm]{Definition}
\theoremstyle{remark}
	\newtheorem{rmk}[thm]{Remark}
	\newtheorem{expl}[thm]{Example}

\newcommand*{\alg}{\mathcal}		
\newcommand*{\hilb}{\mathfrak}		
\newcommand*{\N}{\mathbb N}		
\newcommand*{\Z}{\mathbb Z}		
\newcommand*{\R}{\mathbb R}		
\newcommand*{\C}{\mathbb C}		
\newcommand*{\HS}{\mathrm{HS}}		
\newcommand*{\one}{\mathbbm 1}		
\newcommand*{\tensor}{\otimes}		
\newcommand*{\bigtensor}{\bigotimes}	
\newcommand*{\atensor}{\odot}		
\newcommand*{\B}{\alg B}		
\renewcommand*{\H}{\hilb H}		
\newcommand*{\BH}{\B(\H)}		
\newcommand*{\M}{\alg M}		
\DeclarePairedDelimiter{\abs}{\lvert}{\rvert}	
\DeclarePairedDelimiter{\scal}{\langle}{\rangle}	
\DeclarePairedDelimiter{\norm}{\lVert}{\rVert} 	
\DeclareMathOperator{\lin}{lin}		
\DeclareMathOperator{\supp}{supp}	
\DeclareMathOperator{\Tr}{Tr}		
\DeclareMathOperator{\Id}{Id}		

\newcommand*{\cf}{\mbox{cf.}\xspace}
\newcommand*{\eg}{\mbox{e.\,g.}\xspace}
\newcommand*{\etal}{\mbox{et al.}\xspace}

\newcommand*{\ie}{\mbox{i.\,e.}\xspace}
\newcommand*{\rhs}{\mbox{r.\,h.\,s.}\xspace}

\newcommand*{\ndash}{\nobreakdash-}
\newcommand*{\Star}{$^*$\ndash}

\begin{document}
\maketitle
\begin{abstract}
	In this article we extend the coupling method from classical probability theory to quantum Markov chains on atomic von Neumann algebras. In particular, we establish a coupling inequality, which allow us to estimate convergence rates by analyzing couplings. For a given tensor dilation we construct a self-coupling of a Markov operator. It turns out that on a dense subset the coupling is a dual version of the extended dual transition operator previously studied by Gohm \etal We deduce that this coupling is successful if and only if the dilation is asymptotically complete.
\end{abstract}

\sloppy

\section{Introduction}
In classical probability theory the coupling method has evolved to one of the standard techniques to analyze ergodic properties of stochastic processes. It is particularly useful for two kinds of problems: First, ergodic properties, \eg ergodicity or mixing, can be characterized by the possibility to couple with a certain type of systems (see \eg \cite{Fuerstenberg, Glasner}).%
\footnote{Note that the literature on these problems prefers the term "joining" instead of coupling.}
Second, the mixing time of a process can be estimated by constructing and analyzing particular couplings (see \eg \cite{Pitman74, Lindvall}). 
For the first type of problems some generalization to quantum theory has been successfully investigated in recent years by F.\,Fidaleo \cite{Fidaleo09} and by a series of papers of R.\,Duvenhage, \eg \cite{Duvenhage08, Duvenhage10a, Duvenhage12}.
To our knowledge the second type of problems, however, has not been addressed in quantum theory yet. 

This paper gives a first approach to the coupling method for quantum Markov chains. We establish an upper bound for the speed of mixing of the process via couplings. \mbox{Furthermore,} we provide a construction of couplings to which this estimate can successfully be applied. We show that the convergence of the coupling is intimately related to the scattering theory developed by B.\,K\"ummerer and H.\,Maassen~\cite{Kuemmerer-Maassen00}. For classical systems this was already observed by R.\,Gohm, B.\,K\"ummerer, and T.\,Lang~\cite{Gohm-etal}. We prove a duality that extends their observation to non-commutative systems.

For non-classical systems we observe that the constructed couplings converge to highly entangled states. Although we do not explicitly investigate entanglement properties, we hope that our investigations also sheds some light onto the subject of entangled Markov operators or entangled quantum channels.

The paper itself is organized as follows: After some preliminaries, Section~\ref{sec:couplings} introduces couplings in general and provides some elementary properties of the diagonal coupling. 
Our von Neumann algebraic approach raises the question under which conditions the (purely algebraic) diagonal state admits a normal extension. This is addressed in Section~\ref{sec:extension}, where we find that this requirement restricts us to purely atomic von Neumann algebras. 
In Section~\ref{sec:q_c_inequ} we establish a coupling inequality by means of so called \emph{diagonal projections}. The tightest bounds in this estimate are given by maximal diagonal projections. We characterize these projections for the algebra $\BH$, the bounded operators on a Hilbert space. 

From Section~\ref{sec:diagonal_c} on we turn our attention to Markovian dynamics. For a tensor dilation of a Markov operator we introduce a canonically associated coupling, the so called \emph{diagonal coupling}. We prove in Section~\ref{sec:convergence_diagonal_c} that this coupling satisfies a duality to the extended dual transition operator studied by R.\,Gohm \cite{Gohm04a}. This allows us to give a necessary and sufficient condition for the convergence of the diagonal coupling and hence a criteria for the original Markov chain to be mixing.
We finish by applying our method to some illustrating examples and discussing the result in Section~\ref{sec:applications}.

\section*{Notations and Prelilminaries}
Throughout the paper, complex Hilbert spaces are usually denoted by $\H$ or $\hilb K$. Their scalar product is assumed to be linear in the first component. For their Hilbert space tensor product we write $\H \tensor \hilb K$. The set of bounded operators on a Hilbert space~$\H$ is denoted by~$\BH$. 
Von Neumann algebras are usually denoted by $\M$ or $\alg N$. For two von Neumann algebras $\M$ and $\alg N$ we write $\M \tensor \alg N$ for their von Neumann tensor product and $\M \atensor \alg N$ for their algebraic tensor product.
A unital completely positive normal map $S:\M \to \M$ is briefly called a \emph{Markov operator}. For two such maps $S:\M_1 \to \M_2$ and $T:\alg N_1 \to \alg N_2$ their tensor product is denoted by \mbox{$S \tensor T:\M_1 \tensor \alg N_1 \to \M_2 \tensor \alg N_2$}.

Let $\M$ be a von Neumann algebra. Most of the time we will assume that $\M$ acts in standard representation on its standard Hilbert space $L^2(\M)$. We write \mbox{$J:L^2(\M) \to L^2(\M)$} for the modular conjugation. By Tomita-Takesaki-Theory the commutant of $\M$ in this representation is given by $\M' = J \M J$. Thus, the von Neumann algebra~$\M'$ is anti-linearly isomorphic to $\M$ via the map $x \mapsto JxJ$, or equivalently, $\M'$ is anti-multiplicatively isomorphic to $\M$ via $x \mapsto Jx^*J$. For a normal state $\varphi$ on $\M$ we write $\xi_\varphi$ for its standard GNS-vector. By $\sigma_t^\varphi$, $t \in \R$, we denote the modular automorphism group associated with $\varphi$. For further details on Tomita-Takesaki-Theory we recommend
{\cite{Pedersen} or \cite{Stratila}}.

\pagebreak[3]
For each completely positive map $T:\M \to \alg N$ between von Neumann algebras $\M$ and~$\alg N$ we define its \emph{opposite map} by
\begin{equation*}
	T':\M' \to \alg N', \quad T'(x') := J \, T (J\, x'\, J) \, J.
\end{equation*}
For a normal state $\varphi$ on $\M$ its opposite $\varphi'$ is given by the same GNS-vector, \ie, for all $x \in \M'$ we have
\begin{equation*}
	\varphi'(x') = \scal{x' \xi_\varphi, \xi_\varphi}.
\end{equation*}

The standard representations of the algebra $\B(\H)$ for some Hilbert space $\H$ is given by the Hilbert space  $\HS(\H)$ of all Hilbert-Schmidt operators on $\H$, on which $\B(\H)$ acting via by left multiplication. The modular conjugation is given by $J \rho = \rho^*$, $\rho \in \HS(\H)$. 
We will frequently identify $\HS(\H)$ canonically with the tensor product $\H \tensor \bar\H$, where $\bar \H$ denotes the conjugated Hilbert space of $\H$. In this picture the standard representation is given by $\BH$ acting on the first tensor factor. The modular conjugation is the continuous linear extension of $J(\xi \tensor \eta) = \eta \tensor \xi$ for all $\xi \in \H$, $\eta \in \bar\H$. In this particular example the von Neumann tensor product $\BH \tensor \BH'$ is isomorphic to $\B\bigl(\HS(\H) \bigr) = \B(\H \tensor \bar\H)$.

\section{Coupling of Quantum States \& Diagonal States}
\label{sec:couplings}

In this section we introduce the basic notions of couplings and diagonal states. These notions can be defined in a C$^*$-algebraic or even purely algebraic context. However, we restrict ourselves to von Neumann algebras for reasons that will become important in the later sections: First, we use the Tomita-Takesaki-Theory for von Neumann algebras in some of our results. Second, the coupling inequality established in Section~\ref{sec:q_c_inequ} relies on diagonal projections, whose existence is not guaranteed for C$^*$-algebras. And third, the scattering theory for Markov processes discussed in Section~\ref{sec:convergence_diagonal_c} is developed for von Neumann algebras.

\begin{defn}
	Let $\varphi$ and $\psi$ be normal states on von Neumann algebras $\M$ and $\alg N$, respectively. \mbox{A~\emph{coupling}} of $\varphi$ and $\psi$ is a normal state $\hat \varphi$ on $\M \tensor \alg N$ satisfying
	\begin{align}
		\label{eq:coupl_states}
		\hat \varphi(x \tensor \one) &= \varphi(x) \;,
		&
		\hat \varphi(\one \tensor y) &= \psi(y) 
	\end{align}
	for every $x \in \M$ and $y \in \alg N$. The states $\varphi$ and $\psi$ are called the \emph{marginals} of $\hat \varphi$.
\end{defn}

In the literature on stochastic dynamical systems a coupling of $\varphi$ and $\psi$ with an additional invariance property is also called a \emph{joining}. This term goes back to the work of H.~F\"urstenberg (\eg, \cite{Fuerstenberg}) and refers to the notion of a joint distribution. However, we decided to use the term ``coupling'', to refer to the established coupling method in classical probability theory (see \eg \cite{Lindvall,Thorisson}).

For two fixed states $\varphi$ and $\psi$ the set of all couplings of $\varphi$ and $\psi$ is obviously a convex set containing the product state $\varphi \tensor \psi$. Moreover, $\varphi \tensor \psi$ is the only coupling of $\varphi$ and~$\psi$ if and only if one of the states $\varphi$ or $\psi$ is pure. For the algebra~$M_k$ of \mbox{$(k\times k)$-matrices} the set of all couplings with fixed marginals is frequently studied in the context of entanglement. For more detailed characterizations of this set we refer to \cite{Rudolph, Parthasarathy}.

In classical probability theory for each probability measure $\mu$ on a locally compact Hausdorff space~$\Omega$ there is a (unique) measure on the product space $\Omega \times \Omega$ that is concentrated on the diagonal $\Delta := \{(\omega, \omega) \;|\; \omega \in \Omega \}$ and whose marginals are both equal to $\mu$. This measure is called the \emph{diagonal measure} of $\mu$. It is given by the extension of $\mu_\Delta(A \times B) := \mu(A \cap B)$ for all measurable subsets $A, B\subseteq \Omega$. For a state $\varphi$ on a non-commutative algebra $\M$, however, the straightforward definition via $\varphi_\Delta(a \tensor b) := \varphi(a \cdot b)$ for all $a, b \in \M$ fails to be positive in general. A workaround to deal with this issue was proposed by R.~Duvenhage \cite{Duvenhage08} and F.~Fidaleo \cite{Fidaleo09}. In the following we repeat this construction and derive some elementary properties.

Let $\M$ be a von Neumann algebra and assume that $\M$ is given in its standard representation on $L^2(\M)$. By definition the elements of $\M$ and $\M'$ commute. Therefore, we obtain a representation of the algebraic tensor product $\M \atensor \M'$ on the standard Hilbert space $L^2(\M)$ by putting 
\begin{equation*}
	\pi_\Delta: \M \atensor \M' \to \B \bigl(L^2(\M) \bigr), \quad x \tensor y' \mapsto x \cdot y' 
\end{equation*}
and extending linearly. We call $\pi_\Delta$ the \emph{diagonal representation}. In general $\pi_\Delta$ cannot be extended to a normal representation of $\M \tensor \M'$. In fact, it can be extended normally if and only if the algebra $\M$ is atomic (see \cite{Fidaleo01}).

\pagebreak[3]
\begin{defn}
	Let $\varphi$ be a normal state on $\M$ and denote by $\xi_\varphi \in L^2(\M)$ its GNS-vector. The \emph{diagonal state} of $\varphi$ is the state $\varphi_\Delta$ on the algebraic tensor product $\M \atensor \M'$ given by the linear extension of 
	\begin{equation*}
		\varphi_\Delta(x \tensor y') := \scal[\big]{\pi_\Delta(x \tensor y') \, \xi_\varphi, \xi_\varphi} = \scal{xy' \, \xi_\varphi, \xi_\varphi} 
		\qquad 
		(x \in \M, y' \in \M').
	\end{equation*}
\end{defn}

The diagonal state satisfies $\varphi_\Delta(x \tensor \one) = \varphi(x)$ and $\varphi_\Delta(\one \tensor  y') = \varphi'(y')$ for all $x \in \M$ and \mbox{$y'\in \M'$}. Thus, apart from its algebraic domain, $\varphi_\Delta$ is a coupling of the state~$\varphi$ and its opposite state $\varphi'$. 

\begin{expl}
	\label{expl:diag_state_BH}
	Let $\varphi$ be a normal state on $\BH$ and $\varphi(x) = \sum_{i \in I} \lambda_i \scal{x e_i, e_i}$, $x \in \BH$, a spectral decomposition with an orthonormal system $(e_i)_{i \in I}$ and coefficients $0 \le \lambda_i \le 1$, $\sum_{i \in I} \lambda_i = 1$. Then the normal extension of the diagonal state $\varphi_\Delta$ to $\BH \tensor \BH' = \B(\H \tensor \bar\H)$ is given by 
	\begin{align*}
		\varphi_\Delta(z) &= \scal{z \, \xi_\varphi, \xi_\varphi}
		&
		&\text{with}
		&
		\xi_\varphi &:= \sum\nolimits_{i \in I} \sqrt{\lambda_i} \; e_i \tensor e_i
	\end{align*}
	for every $z \in \B(\H \tensor \bar\H)$. 
\end{expl}

The following proposition states that the map $\varphi \mapsto \varphi_\Delta$ naturally respects the decomposition of the algebra into direct sums and tensor products. The precise formulation though looks a bit technical. We leave the straightforward proof to the reader.

\pagebreak[3]
\begin{prop}	\label{prop:diag_sum_prod}
	Let $\varphi$ and $\psi$ be normal states on von Neumann algebras $\M$ and $\alg N$, respectively.
	\begin{enumerate}
	\item 
		For $0 \le \lambda \le 1$ consider the normal state $\lambda \varphi \oplus (1-\lambda) \psi$ on the direct sum $\M \oplus \alg N$. Then
		\begin{equation*}
			\bigl( \lambda \varphi \oplus (1-\lambda) \psi \bigr)_\Delta = \bigl( \lambda \varphi_\Delta \oplus (1-\lambda) \psi_\Delta \bigr) \circ P \;,
		\end{equation*}
		where $P:(\M \oplus \alg N) \atensor (\M' \oplus \alg N') \to (\M \atensor \M') \oplus (\alg N \atensor \alg N')$ denotes the canonical projection given by the linear extension of $P\bigl( (x\oplus y) \tensor (x' \oplus y') \bigr) := (x \tensor x') \oplus (y \tensor y')$ for all $x \in \M$, $x' \in \M'$, $y \in \alg N$, $y' \in \alg N'$.
	\item
		Consider the tensor product state $\varphi \tensor \psi$ on $\M \tensor \alg N$. Then
		\begin{equation*}
			(\varphi \tensor \psi)_\Delta = (\varphi_\Delta \tensor \psi_\Delta) \circ \sigma \;,
		\end{equation*}
		where $\sigma:(\M \atensor \alg N) \atensor (\M' \atensor \alg N') \to (\M \atensor \M') \atensor (\alg N \atensor \alg N')$ denotes the tensor flip of the inner pair of tensor factors.
	\end{enumerate}
\end{prop}

\pagebreak[3]
\begin{prop}	\label{prop:diag_mod_grp}
	Let $\alpha:\M \to \alg N$ be an injective unital normal \Star homomorphism of von Neumann algebras $\M$ and $\alg N$. Let $\psi$ be a faithful normal state on $\alg N$ and put $\varphi := \psi \circ \alpha$. Suppose that $\alpha$ commutes with the corresponding modular automorphism groups, \ie, $\sigma_t^\psi \circ \alpha = \alpha \circ \sigma_t^\varphi$ for every $t \in \R$. Then 
	\begin{equation*}
		\psi_\Delta \circ (\alpha \tensor \alpha') = \varphi_\Delta \;.
	\end{equation*}
\end{prop}
\begin{proof}
	Let $\xi_\varphi \in L^2(\M)$ and $\xi_\psi \in L^2(\alg  N)$ denote the GNS-vectors of $\varphi$ and $\psi$, respectively. The homomorphism $\alpha$ gives rise to an isometry $v:L^2(\M) \to L^2(\alg N)$ by extending $v(x\xi_\varphi) = \alpha(x)\xi_\psi$ for all $x \in \M$. Since $\alpha$ is $^*$-preserving and commutes with the modular automorphism group, this isometry commutes with the modular operator and the modular conjugation. We therefore get
	\begin{equation*}
		v (JyJ) \xi_\varphi = Jv y\xi_\varphi = J\alpha(y) \xi_\psi = \alpha'(JyJ) \xi_\psi
	\end{equation*}
	for every $y\in \M$. Hence for every $x \in \M$ and $y' \in \M'$ we obtain
	\begin{align*}
		\psi_\Delta \bigl( (\alpha \tensor \alpha')(x \tensor y') \bigr) 
		&= \scal{\alpha'(y')\xi_\psi, \; \alpha(x^*)\xi_\psi}
		= \scal{vy'\xi_\varphi, vx^*\xi_\varphi}
		= \scal{y'\xi_\varphi, x^*\xi_\varphi} 
		\\
		&= \varphi_\Delta(x \tensor y'),
	\end{align*}
	that is, we have $\psi_\Delta \circ (\alpha \tensor \alpha') = \varphi_\Delta$ by linear extension.
\end{proof}

\section{Normal Extensions of the Diagonal State}
\label{sec:extension}

In this section we investigate the question whether for a given normal state $\varphi$ on a von Neumann algebra $\M$ the diagonal state $\varphi_\Delta$ on $\M \atensor \M'$ admits a (necessary unique) normal extension to the von Neumann tensor product $\M \tensor \M'$. It turns out that this requirement essentially restricts us to purely atomic von Neumann algebras, \ie algebras of the form $\M = \bigoplus_{i \in I} \B(\H_i)$ for some index set $I$ and some Hilbert spaces $\H_i$, $i \in I$.

\begin{thm}	\label{thm:normal_extension}
	Let $\varphi$ be a normal state on a von Neumann algebra~$\M$. Denote by $c \in \M$ the central support of $\varphi$. If the diagonal state $\varphi_\Delta$ admits a normal extension to $\M \tensor \M'$, then $c \M$ is purely atomic.
	\label{thm:diagonal_state}
\end{thm}
\begin{rmk}
	\begin{enumerate}
	\item 
		The converse obviously holds, too (\cf Example~\ref{expl:diag_state_BH} and Proposition~\ref{prop:diag_sum_prod}).
	\item
		Our proof essentially refines the arguments given by Fidaleo \cite{Fidaleo01}, who studied normal extensions of the diagonal representation.
	\end{enumerate}
\end{rmk}

For the proof we first establish two lemmas dealing separately with the abelian case and factors. For the abelian case consider a locally compact Hausdorff space $\Omega$ equipped with a Radon measure $\mu$. For a measure $\nu$ on $\Omega$ we write $\nu \ll \mu$ if $\nu$ is absolutely continuous with respect to $\mu$. 

\begin{lemma}	\label{lem:diagonal_Radon}
	Let $\nu \ll \mu$ be a Radon probability measure on $\Omega$ and denote by $\nu_\Delta$ its diagonal measure on $\Omega \times \Omega$. If $\nu_\Delta \ll \mu \tensor \mu$ then $\nu$ is purely atomic. 
\end{lemma}
\begin{proof}
	Decompose $\mu =: \mu_a + \mu_c$ into its purely atomic part~$\mu_a$ and its atomless part $\mu_c$. Then 
	\begin{equation*}
		\mu \tensor\mu = \mu_a \tensor \mu_a + \mu_a \tensor \mu_c + \mu_c \tensor \mu_a + \mu_c \tensor \mu_c .
	\end{equation*}
	On the right hand side only the purely atomic measure $\mu_a \tensor \mu_a$ does not vanish on the diagonal \mbox{$\Delta := \{(\omega, \omega) \;|\; \omega \in \Omega\}$}. Since $\Delta$ contains the support of $\nu_\Delta$, it follows that $\nu_\Delta \ll \mu_a \tensor \mu_a$. Consequently, also $\nu_\Delta$ is purely atomic; and so is its marginal $\nu$.
\end{proof}

\begin{lemma}	\label{lem:diagonal_factor}
	Let $\M$ be a factor and let $\varphi$ be a normal state~$\M$. If  the diagonal state~$\varphi_\Delta$ extends to a normal state on $\M \tensor \M'$, then $\M$ is of Type~I.
\end{lemma}
\begin{proof}
	We assume that $\M$ is given in its standard representation on~$L^2(\M)$. Denote by $\xi_\varphi \in L^2(\M)$ the GNS-vector of $\varphi$. Since $\M \cap \M' = \C \one$, the vector $\xi_\varphi$ is cyclic for the subalgebra jointly generated by $\M$ and~$\M'$. Therefore, the GNS-representation of $\M \tensor \M'$ with respect to $\varphi_\Delta$ is (up to unitary equivalence) a normal extension of the diagonal representation. In particular, the diagonal representation admits a normal extension to the von Neumann tensor product
	\begin{equation*}
		\pi_{\varphi_\Delta}:\M \tensor \M' \longrightarrow \B \bigl(L^2(\M) \bigr).
	\end{equation*}
	Since the factor $\M \tensor \M'$ does not have non-trivial $\sigma$-weakly closed ideals, this representation is faithful and, consequently, $\M \tensor \M'$ is isomorphic to the range
	\begin{equation*}
		\pi_{\varphi_\Delta}(\M \tensor \M') = (\M \cup \M')'' = (\M' \cap \M)' = \B \bigl( L^2(\M) \bigr) .
	\end{equation*}
	The assertion now follows from the fact that $\M$, $\M'$, and $\M \tensor \M'$ all share the same type.
\end{proof}

\begin{proof}[Proof of Theorem~\ref{thm:normal_extension}]
	We denote by $\varphi_\Delta$ the normal extension of the diagonal state. We write $Z(\M)$ for the center of $\M$ and denote by $(\Omega, \mu)$ the corresponding locally compact Hausdorff space with Radon measure satisfying $Z(\M) = L^\infty(\Omega,\mu)$. The restriction of~$\varphi$ to the center gives rise to a Radon probability measure $\nu \ll \mu$. On $Z(\M \tensor \M) = L^\infty(\Omega \times \Omega, \mu \tensor \mu)$ the state $\varphi_\Delta$ yields the corresponding diagonal measure~$\nu_\Delta$. Since $\varphi_\Delta$ is normal on $Z(\M \tensor \M)$, it follows that $\nu_\Delta \ll \mu \tensor \mu$. By Lemma~\ref{lem:diagonal_Radon} the measure $\nu$ is purely atomic, \ie, there are (countably many) minimal projections \mbox{$c_i \in Z(M)$}, $i \in I$, with $c = \sum_{i \in I} c_i$.
	
	Now fix an index $i \in I$ and put $\M_i := c_i \M$ and $\M_i' := c_i \M'$. Due to minimality of $c_i$ the von Neumann algebra $\M_i := c_i \M$ is a factor. On this algebra consider the renormalized normal state $\varphi_i(x) := \varphi(x) / \varphi(c_i)$, $x \in \M_i$. The diagonal state of $\varphi_i$ admits a normal extension to $\M_i \tensor \M_i'$ given~by  
	\begin{equation*}
		(\varphi_i)_\Delta(z) 
		:= \frac{\varphi_\Delta \bigl( (c_i \tensor c_i) z \bigr)}{\varphi_\Delta(c_i \tensor c_i)} 
	\end{equation*}
	for every $z \in \M_{i} \tensor \M_{i}' \subseteq \M \tensor \M'$. With~Lemma~\ref{lem:diagonal_factor} we conclude that $\M_i$ is Type~I; in particular, $\M_i$ is purely atomic. The direct sum $c\M = \bigoplus_{i \in I} \M_i$ is therefore purely atomic, too.
\end{proof}

\section{A Quantum Coupling Inequality}
\label{sec:q_c_inequ}

After introducing the notion of couplings in general, we now show how a coupling can be used to derive estimates for the distance of its marginals. This is well known for probability measures. More precisely, for a coupling (joint distribution) $\hat \mu$ of two probability measures $\mu$ and $\nu$ on some locally compact Hausdorff space $\Omega$ the inequality
\begin{align*}
	\mu(A) - \nu(A) 
	&= \hat\mu(A \times \Omega) - \hat\mu(\Omega \times A)
	\\
	&= \hat\mu \bigl( (A \times \Omega) \cap \Delta^C \bigr) - \hat\mu \bigl( (\Omega \times A) \cap \Delta^C \bigr) 
	\\
	&\le \hat \mu \bigl( (\Omega \times \Omega) \cap \Delta^C \bigr) 
	\le \hat\mu(\Delta^C)
\end{align*}
holds for every measurable set $A \subseteq \Omega$, where $\Delta^C$ denotes the complement of the diagonal $\Delta := \{(\omega, \omega) \;|\; \omega \in \Omega\}$ in~$\Omega \times \Omega$. It follows that the variation distance of $\mu$ and $\nu$ can be estimated by
\begin{equation*}
	\norm{\mu - \nu} \le 2 \bigl( 1- \hat\mu(\Delta) \bigr).
\end{equation*}
This so called \emph{coupling inequality} provides one of the fundamental ingredients of the coupling method in classical probability theory. In this section we show that a similar estimate holds for quantum couplings.

\begin{defn}
	A projection $p_\Delta \in \M \tensor \M'$ is called a \emph{diagonal projection} if 
	\begin{equation*}
		p_\Delta (x \tensor \one) p_\Delta = p_\Delta (\one \tensor Jx^* J) p_\Delta 
	\end{equation*}
	holds for every $x \in \M$.
\end{defn}

Our main motivation for studying diagonal projections is founded on the following simple estimate, which provides a coupling inequality in the non-commutative context:

\begin{thm}[Quantum Coupling Inequality]
	\label{thm:QCI}
	Let $\varphi$ and $\psi$ be normal states on a von Neumann algebra $\M$. Furthermore, let $\hat \varphi$ be a coupling of $\varphi$ and the opposite state $\psi'$ and let $p_\Delta \in \M \tensor \M'$ be a diagonal projection. Then 
	\begin{equation*}
		\norm{\varphi - \psi} \le 4 \bigl( 1-\hat\varphi(p_\Delta) \bigr)^{1/2} .
	\end{equation*}
\end{thm}
\begin{proof}
	Let $x \in \M$ with $\norm x \le 1$. For sake of brevity we put $x_\Delta := x \tensor \one - \one \tensor Jx^* J$. Then we have
	\begin{align*}
		\abs{\varphi(x) - \psi(x)} 
		&= \abs{\hat \varphi(x_\Delta)}
		\le \abs{\hat \varphi(p_\Delta \, x_\Delta \, p_\Delta)} + \abs{\hat \varphi(p_\Delta \, x_\Delta \, p_\Delta^\bot)} + \abs{\hat \varphi(p_\Delta^\bot \, x_\Delta)}.
	\end{align*}
	On the right hand side, the first term of the sum vanishes, because $p_\Delta$ is a diagonal projection. For the second and third term an application of the Cauchy Schwarz inequality shows that both are less than $\hat\varphi(p_\Delta^\perp)^{1/2} \cdot \norm{x_\Delta}$; and hence
	\begin{align*}
		\abs{\varphi(x) - \psi(x)} 
		&\le 2 \cdot \hat\varphi(p_\Delta^\perp)^{1/2} \cdot \norm{x_\Delta}
	\end{align*}
	Since $\norm x \le 1$ implies $\norm{x_\Delta} \le 2$, the assertion of the theorem follows by taking the supremum over all elements $x \in \M$ with $\norm x \le 1$.
\end{proof}

\begin{rmk}
	The estimate stated in the theorem is chosen to have a similar form as the classical inequality. In fact, the above proof yields the slightly better bound 
	\begin{equation}
		\norm{\varphi - \psi} \le 2 \, \bigl( 1 +\hat\varphi(p_\Delta)^{1/2} \bigr) \, \bigl( 1 - \hat\varphi(p_\Delta) \bigr)^{1/2}
		\label{eq:refined_c_ineq}
	\end{equation}
	if we apply the Cauchy Schwarz inequality more carefully. This estimates is not tight in general as the following simple example shows:
\end{rmk}

\begin{expl}
	\label{ex:tightness_c_ineq}
	Consider the algebra $M_2$ of $(2\times 2)$-matrices and the two states 
	\begin{align*}
		\varphi(x) &:= x_{1,1},
		&
		\psi(x) &:= \tfrac12 \scal{x \; \begin{psmallmatrix} 1 \\ 1 \end{psmallmatrix}, \begin{psmallmatrix} 1 \\ 1 \end{psmallmatrix}},
		&
		&x \in M_2.
	\end{align*}
	The norm distance of these states is given by $\norm{\varphi - \psi} = \sqrt2$.	Since $\varphi$ and $\psi'$ are pure, their only coupling is $\hat\varphi = \varphi \tensor \psi'$. For $\hat\varphi(p_\Delta) \le \nicefrac14$ the \rhs of inequality~\eqref{eq:refined_c_ineq} yields a value greater or equal to 2, thus the inequality is trivial. For $\hat\varphi(p_\Delta) > \tfrac14$ the \rhs of the inequality is decreasing as a function of $\hat\varphi(p_\Delta)$. For the tightest bound we therefore have to maximize $\hat\varphi(p_\Delta)$ among all diagonal projections. In Proposition~\vref{prop:max_diag_proj} we will show that all maximal diagonal projections in $M_2 \tensor M_2'$ are of the form
	\begin{equation*}
		p_\Delta := (p \tensor JpJ) + (p^\perp \tensor J p^\perp J)
	\end{equation*}
	for some one-dimensional projection $p \in M_2$. For such a projection $p$ fix a unit vector in the corresponding subspace of the form $\xi = \begin{psmallmatrix} \sqrt r \\ e^{i\omega} \sqrt t \end{psmallmatrix}$ with $r,t \ge 0$, $r+t=1$, $\omega \in \R$. Then we obtain
	\begin{equation*}
		\hat \varphi(p_\Delta) = \tfrac12 \cdot r \cdot \abs{\sqrt r + e^{i\omega} \sqrt t}^2 + \tfrac12 \cdot t \cdot \abs{\sqrt r - e^{i\omega} \sqrt t}^2 = \tfrac12 + \Re(e^{i\omega}) (r-t) \sqrt{rt}.
	\end{equation*}
	The maximum $\varphi(p_\Delta) = \tfrac34$ is attained at $\omega = 0$, $r = \tfrac12 + \tfrac14 \sqrt2$, and $t = \tfrac12 - \tfrac14 \sqrt2$. Hence for every diagonal projection we have $\varphi(p_\Delta) \le \tfrac34$  and, consequently,
	\begin{equation*}
		\norm{\varphi -\psi} = \sqrt2 < 1+ \tfrac12 \sqrt{3} \le 2 \bigl( 1+\hat\varphi(p_\Delta)^{1/2} \bigr) \, \bigl( 1 - \hat\varphi(p_\Delta) \bigr)^{1/2}.
	\end{equation*}
\end{expl}

The maximal diagonal projections play a distinguished role: First, it is immediate from the definition that projections $q \in \M \tensor \M'$ that are dominated by a diagonal projection $p_\Delta \ge q$ are again diagonal. To determine all diagonal projections it hence suffices to characterize the maximal diagonal projections. Second, the maximal projections yield the best bounds in the coupling inequality of Theorem~\ref{thm:QCI}. 

\begin{expl}
	\begin{enumerate}
	\item 
		For the von Neumann algebra $\M = \ell^\infty(S)$ with a discrete set $S$ all diagonal projections $\chi_A \in \ell^\infty(S \times S)$ are characteristic functions of some subset of the diagonal $\Delta := \{(\omega, \omega) \;|\; s\in S\}$. In particular, there is a \emph{unique} maximal diagonal projection, namely $\Delta$ itself. 
	\item
		For a normal state $\varphi$ on $\BH$ denote by $\varphi_\Delta$ the normal extension of the diagonal state to $\BH \tensor \BH' = \B\bigl( \HS(\H)\bigr)$. The one-dimensional support projection $p_\Delta := \supp \varphi_\Delta$ then is a diagonal projection. These projection will play a distinguished role later (Theorem~\ref{thm:ac_diag}). For different normal states $\varphi$ and $\psi$ the two diagonal projections $p_\Delta := \supp \varphi_\Delta$ and $q_\Delta := \supp \psi_\Delta$ in general have no dominating diagonal projection in common. More precisely, it follows from the following proposition that there is a diagonal projection larger than $p_\Delta$ and $q_\Delta$ if and only if $\varphi$ and $\psi$ have commuting densities.
	\end{enumerate}
\end{expl}

\begin{prop}	\label{prop:max_diag_proj}
	For a projection $p_\Delta \in \BH \tensor \BH'$ the following statements are equivalent:
	\begin{equivalence}
	\item
		$p_\Delta$ is a maximal diagonal projection.
	\item
		There is an orthonormal basis $(e_i)_{i \in I}$ of $\H$ such that
		\begin{equation*}
			p_\Delta = \sum_{i \in I} p_i \tensor J p_i J, 
		\end{equation*}
		where $p_i$ denotes the projection onto $\C e_i$ and where sum converges in the $\sigma$-strong topology.
	\end{equivalence}
\end{prop}
\begin{proof}
	We assume that $\BH$ is given in its standard representation acting via left multiplication on the Hilbert space $L^2(\H)$ of all Hilbert-Schmidt operators on $\H$. 
	\begin{enumerate}
	\item	\label{en:max_diag_proj_vect}
		As a first step we show that a one-dimensional projection onto the subspace $\C\rho$, $\rho \in \HS(\H)$, is diagonal if and only if $\rho$ is a normal Hilbert-Schmidt operator:
		Take an arbitrary unit vector $\rho \in L^2(\H)$ and denote by $p_\rho \in \B \bigl( L^2(\H) \bigr)$ the orthogonal projection onto the subspace $\C \rho \subseteq L^2(\H)$. Since $p_\rho \cdot z \cdot p_\rho = \scal{z \rho, \rho} \cdot p_\rho$ for every $z \in \B\bigl(\HS(\H) \bigr)$, it follows that $p_\rho$ is a diagonal projection if and only if the equation
		\begin{equation}
			\label{eq:max_diag_proj_1}
			\scal{(x \tensor \one) \rho, \rho} = \scal{(\one \tensor Jx^*J) \rho, \rho}
		\end{equation}
		holds for every $x \in \BH$. Computing the left and right hand side, we obtain
		\begin{align*}
			\scal{(x \tensor \one) \rho, \rho} 
			&= \Tr(\rho^* x \rho) 
			= \Tr(\rho \rho^* x),
			&
			\scal{(\one \tensor Jx^* J)\rho, \rho}
			&= \Tr(\rho^* \rho x)
		\end{align*}
		for every $x \in \BH$. Now, if $\rho^*\rho = \rho \rho^*$ then both sides of \eqref{eq:max_diag_proj_1} obviously coincide, that is, $p_\rho$ is a diagonal projection. Conversely, since the functionals $\Tr(\cdot \; x)$ with $x \in \BH$ are separating for $\BH$, Equation~\ref{eq:max_diag_proj_1} implies $\rho^* \rho = \rho \rho^*$.
	\item	\label{en:max_diag_proj_commute}
		As a second step we show that for a diagonal projection $p_\Delta$ the corresponding subspace $p_\Delta \HS(\H)$ contains mutually commuting Hilbert-Schmidt operators:
		Let $\rho_1, \rho_2 \in p_\Delta L^2(\H)$ be two elements of $p_\Delta \HS(\H)$. For each $\alpha \in \C$ the element $\rho_1 + \alpha \rho_2$ is in this subspace, too. Since the projection onto $\C (\rho_1 + \alpha \rho_2)$ is dominated by $p_\Delta$, it is also a diagonal projection. From part~\ref{en:max_diag_proj_vect} of the proof we know that $\rho_1 + \alpha \rho_2$ is a normal operator for all $\alpha \in \C$. By a polarization argument it follows that $\rho_1 \rho_2^* = \rho_2^* \rho_1$ and hence $\rho_1 \rho_2 = \rho_2 \rho_1$ (\cf~\cite[Ex.~IX.2.18]{Conway}).
	\item
		Finally we prove the assertion:
		Let $p_\Delta \in \B \bigl( L^2(\H) \bigr)$ be a maximal diagonal projection. By part~\ref{en:max_diag_proj_commute} of the proof the corresponding subspace $p_\Delta L^2(\H)$ contains mutually commuting normal Hilbert-Schmidt operators. Consequently, they are simultaneously diagonalizable, that is, there is an orthonormal basis $(e_i)_{i \in I}$ of $\H$ such that for each $\rho \in p_\Delta L^2(\H)$ all $e_i$, $i \in I$, are eigenvectors of $\rho$. Put
		\begin{equation*}
			q_\Delta := \sum_{i \in I} p_i \tensor Jp_iJ \in \B \bigl( L^2(\H) \bigr),
		\end{equation*}
		where $p_i$ denotes the projection onto $\C e_i$. It is straightforward to check that this is indeed a diagonal projection.  By construction we have $q_\Delta \ge p_\Delta$ and hence $p_\Delta = q_\Delta$ by maximality.
		\qedhere
	\end{enumerate}
\end{proof}

\section{The Diagonal Coupling of a Tensor Dilation}
\label{sec:diagonal_c}

From states we now shift out attention to unital completely positive normal maps, briefly called \emph{Markov operators}. Our goals is to apply the Quantum Coupling Inequality to the (in general irreversible) dynamical system given such a map. For this purpose we study not only couplings of states, but also couplings of Markov operators in the following sense:

\begin{defn}
	Let $S:\M \to \M$ and $T:\alg N\to \alg N$ be two Markov operators on von Neumann algebras $\M$ and $\alg N$, respectively. A~\emph{coupling} of $S$ and $T$ is a Markov operator $\hat T:\M \tensor \alg N \to \M \tensor \alg N$ such that
	\begin{align*}
		\hat T(x \tensor \one) &= S(x) \tensor \one,
		&
		\hat T(\one \tensor y) &= \one \tensor T(y)
	\end{align*}
	holds for every $x \in \M$ and $y \in \alg N$.
\end{defn}
As previously for states, the set of all couplings of two fixed Markov operators clearly is a convex set containing the tensor product map $S \tensor T$. Moreover, $S \tensor T$ is the only coupling if and only if $S$ or $T$ is an extremal Markov operator (see \cite{Haapasalo-etal}). In contrast to couplings of states not every Markov operator $\hat T:\M \tensor  \alg N \to \alg M \tensor \alg N$ is a coupling; it has to leave the algebras $\M \tensor \one$ and $\one \tensor \alg N$ invariant. 

\begin{expl}	\label{ex:M_c_convex}
	Suppose $S$ and $T$ admit convex decompositions $S = \sum_{k=1}^N \lambda_k S_k$ and \mbox{$T = \sum_{k=1}^N \lambda_k T_k$} with Markov operators $S_k$ and $T_k$, and with the same coefficients \mbox{$0 \le \lambda_k \le 1$}, $\sum_{k=1}^N \lambda_k=1$. Then 
	\begin{equation*}
		\hat T := \sum_{k=1}^N \lambda_k \cdot (S_k \tensor T_k)
	\end{equation*}
	is a coupling of $S$ and $T$. If one of the algebras $\M$ or $\alg N$ is commutative and finite dimensional, then every coupling of $S$ and $T$ is of this form. 
\end{expl}

In physics many examples of Markov operators arise from an interaction of the system with a random environment. Mathematically speaking, the Markov operator is given by a dilation. For this paper we restrict ourselves to the investigation of tensor dilations with atomic algebras. In the following we show how such a dilation canonically gives rise to a coupling called the \emph{diagonal coupling}.

Let $T:\M \to \M$ be a Markov operator on a von Neumann algebra~$\M$. Let $\alg C$ be a purely atomic von Neumann algebra and let $\psi$ be a faithful normal state on $\alg C$. Denote by $\psi_\Delta$ the normal extension of the diagonal state of~$\psi$. Furthermore, let \mbox{$\Gamma:\alg M \to \M \tensor \alg C$} be an injective unital normal \Star homomorphism such that
\begin{equation*}
	T = (\Id \tensor \psi) \circ \Gamma,
\end{equation*}
where $\Id \tensor \psi$ denotes the normal linear extension of $x \tensor c \mapsto \psi(c) \cdot x$ for all $x \in \M$ and $c \in \alg C$. 
We denote by $T':\M' \to \M'$ and $\Gamma':\M' \to \M' \tensor \alg C'$ the opposite maps of $T$ and~$\Gamma$, respectively,  and define an injective unital normal \Star homomorphism by
\begin{equation*}
	\hat\Gamma:\M \tensor \M' \longrightarrow (\M \tensor \M') \tensor  (\alg C \tensor \alg C'), 
	\qquad
	\hat\Gamma := \sigma \circ (\Gamma \tensor \Gamma'),
\end{equation*}
where $\sigma:\M \tensor \alg C \tensor  \alg M' \tensor  \alg  C' \to \M \tensor \M' \tensor \alg C \tensor \alg C'$ denotes the tensor flip of the inner pair of tensor factors. Then a straightforward computation yields the following: 

\begin{prop}	\label{prop:tensor_coupling}
	Let $\hat \psi$ be a coupling of $\psi$ and $\psi'$. Then the Markov operator
	\begin{equation*}
		\hat T:\M \tensor \M' \to \M \tensor \M',
		\qquad
		\hat T := (\Id \tensor \hat \psi) \circ \hat \Gamma
	\end{equation*}
	is a coupling of $T$ and $T'$.	
\end{prop}

\begin{defn}
	The Markov operator obtained for the particular choice $\hat\psi := \psi_\Delta$ in Proposition~\ref{prop:tensor_coupling},
	\begin{equation*}
		\hat T_\Delta:\M \tensor \M' \to \M \tensor \M', 
		\qquad
		\hat T_\Delta := (\Id \tensor \psi_\Delta) \circ \hat \Gamma,
	\end{equation*}
	is called the \emph{diagonal coupling} of $T$ associated with $\Gamma$ and~$\psi$.
\end{defn}

\begin{rmk}
	\label{rmk:alg_diag_c}
	If the algebra $\alg C$ is not purely atomic, the same construction still provides us with the algebraic variant of the diagonal coupling $\hat T_\Delta:\M \atensor \M' \to \M \atensor \M'$
\end{rmk}

\section{Convergence of the Diagonal Coupling}
\label{sec:convergence_diagonal_c}

For a moment suppose that $\M$ is purely atomic and that there is a faithful normal state $\varphi$ on $\M$ with $(\varphi \tensor \psi) \circ \Gamma = \varphi$ such that $\Gamma$ commutes with the modular automorphism groups. Then by Proposition~\vref{prop:diag_mod_grp} it follows that the normal extentension of the diagonal state $\varphi_\Delta$ is invariant for the diagonal coupling, \ie,
\begin{equation*}
	\varphi_\Delta \circ \hat T_\Delta = \varphi_\Delta.
\end{equation*}
It follows that the support projection $p_\Delta := \supp \varphi_\Delta$ satisfies $\hat T_\Delta(p_\Delta) \ge p_\Delta$. Thus, the question arises under which conditions the increasing sequence $\hat T_\Delta^n(p_\Delta)$ converges to $\one$.  
It turns out that this question is related to the scattering theory of the dilation established in \cite{Kuemmerer-Maassen00}. To make this connection precise we first briefly recall some fundamentals of this theory: 

Let $\M$ and $\alg C$ be arbitrary von Neumann algebras, let $\Gamma:\M \to \M \tensor \alg C$ an injective unital normal \Star homomorphism, and let $\psi$ be a faithful normal state on $\alg C$. 
We denote by $v:L^2(\M) \to L^2(\M) \tensor L^2(\alg C)$ the unique isometry with
\begin{equation*}
	v\, x\xi_\varphi  = \Gamma(x) (\xi_\varphi \tensor \xi_\psi)
\end{equation*}
for every $x \in \M$.
Let $\alg C^+ := \bigtensor_{n=1}^\infty (\alg C, \psi)$ be the infinite tensor product von Neumann algebra along the product states $\psi^{\tensor n}$ on $\alg C^{\tensor n}$, and write $\psi^+ := \bigtensor_{n=1}^\infty \psi$ for the normal product state on $\alg C^+$ (\cf \cite[Sec.~2.7.2]{Bratteli-Robinson}). Then $\Gamma:\M \to \M \tensor \alg C$ extends to an injective unital normal \Star homomorphism $\alpha:\M \tensor \alg C^+ \to \M \tensor \alg C^+$ with
\begin{equation*}
	\alpha(x \tensor c) := \Gamma(x) \tensor c
\end{equation*}
for every $x \in \M$ and $c \in \alg C^+$. For sake of brevity we put $\varphi^+ := \varphi \tensor \psi^+$ and write \mbox{$\norm{z}_{\varphi^+}:= (\varphi \tensor \psi^+)(z^*z)^{1/2}$} for the norm on $\M \tensor \alg C^+$ induced by $\varphi^+$. 

\begin{defn}
	The \emph{extended dual transition operator} is the Markov operator 
	\begin{equation*}
		Z':\B \bigl(L^2(\M) \bigr) \to \B \bigl( L^2(\M) \bigr), 
		\quad 
		Z'(x) := v^* (x \tensor \one) v.
	\end{equation*}
\end{defn}

The extended dual transition operator was introduced and studied in \cite{Gohm04a} to classify different tensor dilations of $T$. Furthermore, in \cite{Gohm-etal} the authors showed that this operator allows to give a simple characterization of asymptotic completeness of the dilation as follows:

\pagebreak[3]
\begin{thm}[\cf \cite{Gohm-etal, Kuemmerer-Maassen00}]	\label{thm:asymp_compl}
	The following statements are equivalent:
	\begin{equivalence}
	\item	\label{en:asymp_compl_closure}
		The $\sigma$-weak closure of the set $\bigcup_{n \in \N} \alpha^{-n}(\one \tensor \alg C^+)$ contains $\M \tensor \one$.
	\item	\label{en:asymp_compl_norm}
		For every $x \in \M$
		\begin{equation*}
			\lim_{n \to \infty} \norm[\big]{\alpha^n(x \tensor \one) - (\varphi \tensor \Id) \bigl( \alpha^n(x \tensor \one) \bigr)}_{\varphi^+} = 0.
		\end{equation*}
	\item	\label{en:asymp_compl_Z}
		The fixed space of $Z'$ is $\C \one$.
	\end{equivalence}
\end{thm}
\pagebreak[3]
\begin{defn}
	The homomorphism $\Gamma$ is called \emph{asymptotically complete} if one of the equivalent conditions in Proposition~\ref{thm:asymp_compl} is met.
\end{defn}
\begin{rmk}[to the proof of Proposition~\ref{thm:asymp_compl}]
	The equivalence of the conditions \ref{en:asymp_compl_norm} and \ref{en:asymp_compl_Z} was literally proven in \cite[Thm.\,4.3]{Gohm-etal}. The equivalence to condition~\ref{en:asymp_compl_closure} was shown in \cite{Kuemmerer-Maassen00} for reversible dilations, \ie, if $\Gamma$ is of the form $\Gamma(x) = \beta(x \tensor \one)$ for some normal \Star automorphism $\beta$ of $\M \tensor \alg C$. The reader may easily adopt the proof without requiring reversibility. 
\end{rmk}

As a next step we want to show that the extended transition operator and the diagonal coupling are in duality to each other, which allows us to answer the initial question of this section. To describe this duality, we fix a number~$0 \le \alpha \le \nicefrac12$. We denote by $\Delta_\varphi$ the modular operator associated with the state $\varphi$ and define an injective linear map
\begin{equation*}
	i_\alpha:\M \to L^2(\M), \quad i_\alpha(x) := \Delta_\varphi^\alpha x \xi_\varphi.
\end{equation*}
Via this map we may identify $\M$ with a dense subset of $L^2(\M)$. For each $x \in \M$ and $y' \in \M'$ we obtain a linear functional $\Phi_\alpha(x \tensor y')$ on $\B \bigl( L^2(\M) \bigr)$ by putting
\begin{equation}
	\label{eq:Phi-map}
	\Phi_\alpha(x \tensor y')(t) := \scal{t \,i_\alpha(x), \; i_{\frac12 - \alpha}(Jy'J)}
\end{equation}
for every $t \in \B \bigl( L^2(\M) \bigr)$. The linear extension of Equation~\eqref{eq:Phi-map} then gives rise to an injective linear map 
\begin{equation*}
	\Phi_\alpha:\M \atensor \M' \to \B \bigl( L^2(\M) \bigr)_*\;.
\end{equation*}

\begin{thm}
	\label{thm:diag_vs_Z}
	Suppose that $\Gamma$ commutes with the modular automorphism groups of $\varphi$ and $\psi$, \ie, \mbox{$(\sigma_t^\varphi \tensor \sigma_t^\psi) \circ \Gamma = \Gamma \circ \sigma_t^\varphi$} every $t\in \R$. 
	Denote by $\hat T_\Delta:\M \atensor \M' \to \M \atensor \M'$ the diagonal coupling (\cf Remark~\ref{rmk:alg_diag_c}) and by $Z':\B \bigl( L^2(\M) \bigr) \to \B \bigl( L^2(\M) \bigr)$ the extended dual transition operator. Then for every $z \in \M \atensor \M$ we have
	\begin{equation*}
		\Phi_\alpha \bigl( \hat T_\Delta(z) \bigr) = \Phi_\alpha(z) \circ Z' .
	\end{equation*}
\end{thm}

\begin{rmk}
	Notice that the map $\Phi_\alpha$ and the extended dual transition operator $Z'$ depend on the choice of the state $\varphi$ on $\M$, but the diagonal coupling $\hat T_\Delta$ does not.
\end{rmk}

\begin{proof}
	Denote by $\xi_\psi \in L^2(\alg C)$ the GNS-vector and by $\Delta_\psi$ the modular operator associated with~$\psi$.  In analogy to the construction of $i_\alpha$ and $\Phi_\alpha$ we obtain linear injections \mbox{$j_\alpha:\alg C \to L^2(\alg C)$} and $\Psi_\alpha:\alg C \atensor \alg C' \to \B \bigl( L^2(\alg C) \bigr)_*$. For the functional $\Psi_\alpha(c \tensor d')$ with $c \in \alg C$ and $d' \in \alg C'$ we compute
	\begin{align*}
		\Psi_\alpha(c \tensor d')(\one) 
		&= \scal{j_\alpha(c), j_{1/2-\alpha}(Jd'J)} 
		= \scal{\Delta_\psi^\alpha c \xi_\psi, \; \Delta_\psi^{1/2-\alpha} Jd' \xi_\varphi} 
		\\
		&= \scal{c \xi_\varphi, (d')^* \xi_\varphi}
		= \psi_\Delta(c \tensor d').
	\end{align*}
	This equation then extends linearly to the (algebraic) tensor product $\alg C \tensor \alg C'$.
	Since $\Gamma$ commutes with the modular automorphism groups, the associated isometry \mbox{$v:L^2(\M) \to L^2(\M) \tensor L^2(\alg C)$} commutes with all powers of the modular operator on the respective domain. In particular, we have 
	\begin{equation*}
		v i_\alpha(x) 
		= v \Delta_\varphi^\alpha x \xi_\varphi 
		= (\Delta_\varphi^\alpha \tensor \Delta_\psi^\alpha) v x\xi_\varphi 
		= (\Delta_\varphi^\alpha \tensor \Delta_\psi^\alpha) \Gamma(x) \xi_\varphi 
		= (i_\alpha \tensor j_\alpha) \bigl( \Gamma(x) \bigr)
	\end{equation*}
	for every $x  \in \M$. It follows that for every $x \in \M$ and $y' \in \M'$ we obtain
	\begin{align*}
		\MoveEqLeft
		(\Phi_\alpha(x \tensor y') \circ Z')(t)
		= \scal{(t \tensor \one)v i_\alpha(x), \; v i_{1/2-\alpha}(Jy'J)}
		\\
		&= \scal{(t\tensor \one) (i_\alpha \tensor j_\alpha) \bigl( \Gamma(x) \bigr), \; (i_\alpha \tensor j_\alpha) \bigl( J \Gamma'(y') J \bigr)}
		= (\Phi_\alpha \tensor \Psi_\alpha) \bigl( \hat \Gamma(x \tensor y') \bigr)( t \tensor \one)
		\\
		&= \Phi_\alpha \Bigl( (\Id \tensor \psi_\Delta) \bigl( \hat\Gamma(x \tensor y') \bigr) \Bigr)(t)
		= \Phi_\alpha \bigl( \hat T_\Delta(x \tensor y') \bigr),
	\end{align*}
	where $\hat \Gamma:\M \tensor  \M' \to (\M \tensor \M') \tensor (\alg C \tensor \alg C')$ denotes the composition of $\Gamma \tensor \Gamma'$ with the tensor flip of the two inner factors. Linear extension of this equation finally yields the assertion.
\end{proof}

\pagebreak[3]
\begin{thm}
	\label{thm:ac_diag}
	Let $\M = \BH$ and suppose that the hypotheses of Theorem~\ref{thm:diag_vs_Z} are met. Let $p_\Delta \in \M \tensor \M'$ denote the support projection of $\varphi_\Delta$. 
	Then the following statements are equivalent:
	\begin{equivalence}
	\item	\label{en:ac_diag_Gamma}
		$\Gamma$ is asymptotically complete.
	\item	\label{en:ac_diag_proj}
		The (increasing) sequence $\hat T^n_\Delta(p_\Delta)$, $n \in \N$, converges to $\one$ $\sigma$\ndash strongly.
	\item	\label{en:ac_diag_absorb}
		The state $\varphi_\Delta$ is uniformly absorbing for $\hat T_\Delta$, \ie, for every normal state $\hat \varphi$ on~\mbox{$\M \tensor \M'$} we have
		\begin{equation*}
			\lim_{n \to \infty} \norm{\hat \varphi \circ \hat T_\Delta^n - \varphi_\Delta} = 0.
		\end{equation*}
	\item	\label{en:ac_diag_fix}
		The fixed space of $\hat T_\Delta$ is $\C \one$.
	\end{equivalence}
\end{thm}

\pagebreak[3]
\begin{proof}
	The implications \ref{en:ac_diag_absorb}$\Rightarrow$\ref{en:ac_diag_fix}$\Rightarrow$\ref{en:ac_diag_proj} are obvious. According to Example~\ref{expl:diag_state_BH} the diagonal projection $p_\Delta$ is one-dimensional in $\BH \tensor \BH' = \B \bigl( \HS(\H) \bigr)$. Hence the implication \ref{en:ac_diag_proj}$\Rightarrow$\ref{en:ac_diag_absorb} is also well-known (see \eg \cite[Lem.\,A.5.3]{Gohm04a}). 

	We now show that \ref{en:ac_diag_Gamma} implies \ref{en:ac_diag_proj}: Suppose that $\Gamma$ is asymptotically complete, \ie, the extended dual transition operator $Z': \B \bigl( L^2(\M) \bigr) \to \B \bigl( L^2(\M) \bigr)$ has trivial fixed space. The vector state of the GNS-vector $\xi_\varphi \in L^2(\M)$ is invariant for $Z'$. For the one-dimensional projection $p_\varphi$ onto $\C\xi_\varphi$ this implies that $(Z')^n(p_\varphi)$ increases and converges to $\one$ $\sigma$-weakly, \ie, we have
	\begin{equation}
		\label{eq:ac_diag}
		\lim_{n \to \infty} \bigl( \rho  \circ (Z')^n \bigr)(p_\varphi)  = \rho(\one)
	\end{equation}
	for every normal functional $\rho$ on $\B \bigl( L^2(\H) \bigr)$. A~straightforward computation on elementary tensors verifies 
	\begin{align*}
		\Phi_\alpha(z)(\one) &= \varphi_\Delta(z), 
		&
		\Phi_\alpha(z)(p_\varphi) &= (\varphi \tensor \varphi')(z)
	\end{align*}
	for every $z \in \BH \atensor \BH'$. Using the duality of Theorem~\ref{thm:diag_vs_Z}, Equation~\eqref{eq:ac_diag} then implies 
	\begin{equation}
		\label{eq:ac_diag_alg}
		\lim_{n \to \infty} (\varphi \tensor \varphi') \bigl( \hat T_\Delta^n(z) \bigr) = \varphi_\Delta(z)
	\end{equation}
	for every $z \in \BH \atensor \BH'$. Since the Markov operator $\hat T_\Delta$ is norm contractive, this equation extends to all elements $z \in \BH \tensor \BH' = \BH \tensor \B(\bar\H)$ contained in the norm closure of $\BH \atensor \B(\bar\H)$. This closure covers all compact operators on $\H \tensor \bar\H$, in particular the one-dimensional projection $z = p_\Delta$. We therefore have
	\begin{equation*}
		\lim_{n \to \infty} (\varphi \tensor \varphi') \bigl( \hat T_\Delta^n(p_\Delta) \bigr) = 1.
	\end{equation*}
	Since $\varphi \tensor \varphi'$ is a faithful state, it follows that the increasing sequence $\hat T_\Delta^n(p_\Delta)$, $n \in \N$, converges $\sigma$-strongly to $\one$.

	For the converse we essentially reverse the arguments: Suppose condition~\ref{en:ac_diag_absorb} holds. Then we particularly have Equation~\eqref{eq:ac_diag_alg} for every $z \in \BH \atensor \BH'$. It follows that Equation~\eqref{eq:ac_diag} holds for every functional $\rho$ of the form $\rho = \Phi_\alpha(z)$ with $z \in \BH \tensor \BH'$. Since these functionals are norm dense in the predual of $\B \bigl( \HS(\H) \bigr)$, Equation~\eqref{eq:ac_diag} extends to all functionals on $\B \bigl( \HS(\H) \bigr)$. Since $p_\varphi$ is a one-dimensional projection, we finally conclude that $Z'$ has trivial fixed points, \ie, $\Gamma$ is asymptotically complete.
\end{proof}

\section{Applications}
\label{sec:applications}

In this final section we present a classical and a non-commutative example serving different purposes. The classical example demonstrates how the coupling method can be used to bound the speed of convergence by means of coding theory and combinatorics. The example itself is design to be easily accessible by examining suitable pictures, but the method itself immediately generalizes to arbitrary classical Markov processes an finite sets.

For non-commutative examples we do not have the combinatorics available. To apply the coupling method we therefore take a different approach. Although we do not arrive at good bounds for the speed of convergence, for certain examples our method has the advantage of being computationally simple.

\subsection{A Commutative Example}
\label{sec:classical_expl}

\begin{figure}
	\centering
	\begin{tikzpicture}[->, xscale=1.5]
		\useasboundingbox (-1.5,-1) rectangle (3.5,1.5);

		\node[circle, draw] (R) at (0,0) {$s_1$};
		\node[circle, draw] (S) at (1,0) {$s_2$};
		\node[circle, draw] (T) at (2,0) {$s_3$};

		\path 
		(R) 	edge[out=160, in=200, loop, color=red] node[left]{$r$} (R)
			edge[out=130, in=230, loop, color=green] node[left]{$g$} (R)
			edge[bend right, color=blue] node[below]{$b$} (S)
		(S) 	edge[bend right, color=red] node[above]{$r$} (R)
			edge[bend right, color=blue] node[below]{$b$} (T)
			edge[out=70, in=110, loop, color=green] node[above]{$g$} (S)
		(T) 	edge[bend right, color=red] node[above]{$r$} (S) 
			edge[out=340, in=20, loop, color=blue] node[right]{$b$} (T)
			edge[out=310, in=50, loop, color=green] node[right]{$g$} (T);
	\end{tikzpicture}
	\caption{Chosen Road Coloring}
	\label{fig:road_3}
\end{figure}
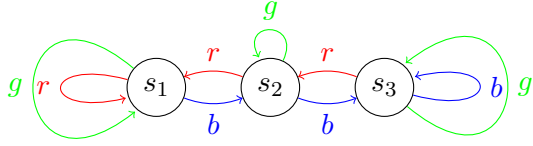
Consider the set $S := \{s_1, s_2, s_3\}$ and the transition matrix
\begin{align*}
	T &:= \begin{pmatrix}
		\nicefrac56 & \nicefrac16 & 0 \\
		\nicefrac13 & \nicefrac12 & \nicefrac16 \\
		0 & \nicefrac13 & \nicefrac23
	\end{pmatrix}
	&
	&\text{with transition graph}
	&
	&\begin{tikzpicture}[->, xscale=1.5, baseline]
		\node[circle, draw] (T) at (0,0) {$s_1$};
		\node[circle, draw] (S) at (1,0) {$s_2$};
		\node[circle, draw] (R) at (2,0) {$s_3$};
		\path 
		(T)	edge[out=160, in=200, loop] node[above]{\nicefrac56} (T)
			edge[bend right] node[below]{\nicefrac16} (S) 
		(S) 	edge[bend right] node[below]{\nicefrac16} (R)
			edge[bend right] node[above]{\nicefrac13} (T)
			edge[out=70, in=110, loop] node[above]{\nicefrac12} (S)
		(R) 	edge[out=340, in=20, loop] node[above]{\nicefrac23} (R)
			edge[bend right] node[above]{\nicefrac13} (S);
	\end{tikzpicture} \;.
\end{align*}
A dilations of $T$ is given by a map $\gamma:S \times C \to S$ with a set $C$ together with a suitable probability measure on the set $C$. Here we choose the set $C := \{r,g,b\}$ (red, green, blue) and the map $\gamma$ pictured in Figure~\vref{fig:road_3}. As a probability measure on $C$ we fix
\begin{align*}
	\nu(r) &:= \nicefrac13,
	&
	\nu(g) &:= \nicefrac12,
	&
	\nu(b) &:= \nicefrac16.
\end{align*}
The diagonal coupling $\hat T_\Delta$ then has a dilation given by the so called \emph{graph product} map
\begin{equation*}
	\hat\gamma: (S \times S) \times C \longrightarrow S \times S, 
	\quad
	\hat\gamma \bigl( (s,s'), \; c \bigr) := \bigl( \gamma(s,c), \; \gamma(s',c) \bigr),
\end{equation*}
where $C$ is again equipped with the measure $\nu$. This map $\hat\gamma$ is shown in Figure~\vref{fig:graph-product}.
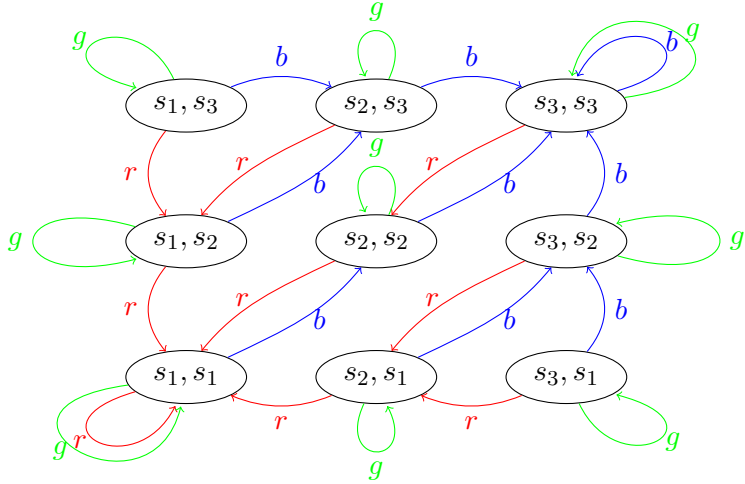
\begin{figure}
	\centering
	\begin{tikzpicture}[->, xscale=2.5, yscale=1.8]
		\useasboundingbox (0,0) rectangle (4,4);

		\foreach \x in {1,2,3} 
			\foreach \y in {1,2,3} 
			{
				\node[ellipse, draw] (s\x\y) at (\x,\y) {$s_\x, s_\y$};
			}

		\foreach \xa/\ya in {1/2, 2/3} \foreach \xb/\yb in {1/2, 2/3}
		{
			\path (s\xa\xb) edge[out=30, in=240, color=blue] node[right]{$b$} (s\ya\yb);
		}
		\foreach \xb/\yb in {1/2, 2/3} 
		{
			\path (s3\xb) edge[bend right, color=blue] node[right]{$b$} (s3\yb);
			\path (s\xb3) edge[bend left, color=blue] node[auto]{$b$} (s\yb3);
		}	
		\path (s33) edge[out=20, in=70, loop, color=blue] node[right]{$b$} (s33);

		\foreach \xa/\ya in {3/2, 2/1} \foreach \xb/\yb in {3/2, 2/1}
		{
			\path (s\xa\xb) edge[out=210, in=60, color=red] node[left]{$r$} (s\ya\yb);
		}
		\foreach \xb/\yb in {3/2, 2/1}
		{
			\path (s1\xb) edge[bend right, color=red] node[left]{$r$} (s1\yb);
			\path (s\xb1) edge[bend left, color=red] node[auto]{$r$} (s\yb1);
		}	
		\path (s11) edge[out=200, in=250, loop, color=red] node[left]{$r$} (s11);

		\path (s11) edge[out=190, in=260, loop, color=green] node[left]{$g$} (s11);
		\path (s12) edge[out=160, in=200, loop, color=green] node[left]{$g$} (s12);
		\path (s13) edge[out=110, in=160, loop, color=green] node[left]{$g$} (s33);
		\path (s21) edge[out=250, in=290, loop, color=green] node[below]{$g$} (s21);
		\path (s22) edge[out=70, in=110, loop, color=green] node[above]{$g$} (s22);
		\path (s23) edge[out=70, in=110, loop, color=green] node[above]{$g$} (s23);
		\path (s31) edge[out=290, in=340, loop, color=green] node[right]{$g$} (s31);
		\path (s32) edge[out=340, in=20, loop, color=green] node[right]{$g$} (s32);
		\path (s33) edge[out=10, in=80, loop, color=green] node[right]{$g$} (s33);
	\end{tikzpicture}
	\caption{Graph Product of $\gamma$}
	\label{fig:graph-product}
\end{figure}

Now, we call a finite sequence $(c_1, \dots, c_n)$ in $C$ a \emph{synchronizing sequence} if all paths in the graph of Figure~\ref{fig:road_3} labeled precisely by $(c_1, \dots, c_n)$ share the same destination (\cf~\cite{Gohm-etal}). In the graph product, a path labeled by a synchronizing word obviously has its destination on the diagonal $\Delta := \{(s,s) \;|\; s \in S \}$. We therefore have the estimate
\begin{equation*}
	(\hat\mu \circ \hat T_\Delta^n)(\Delta) \ge \nu^{\tensor n} \bigl( \{(c_1, \dots, c_n) \;|\; (c_1,\dots, c_n) \;\text{synchronizing} \bigr)
\end{equation*}
for an arbitrary probability measure $\hat\mu$ on~$S\times S$.

Inspecting Figure~\ref{fig:graph-product} it is immediate that the synchronizing sequences are precisely those that contain $(r,g,\dots, g,r)$ or $(b,g,\dots, g,b)$ as a subsequence. Equivalently, the \textbf{non}\ndash synchronizing sequences are precisely those that after removing all $g$'s are alternating in $r$ and $b$. For each of this sequences its probability with respect to $\nu^{\tensor n}$ only depends on the number of occurrences of $r$, $b$, and $g$. Since there are $\binom{n}{k}$ different sequences of length $n$ that contain $n-k$ times the element $g$, we arrive at the estimate
\begin{align*}
	(\hat \mu \circ \hat T^n)(\Delta) 
	&\ge 1 - \sum_{k=0}^n \binom{n}{k} (\tfrac12)^{n-k} \cdot \begin{cases}
		1 &\text{, if $k=0$}
		\\
		2 \cdot (\tfrac13 \cdot \tfrac16)^{k/2} &\text{, if $0 \neq k$ is even} 
		\\
		(\tfrac16 + \tfrac13) \cdot (\tfrac13 \cdot \tfrac16)^{(k-1)/2} &\text{, if $k$ is odd}
	\end{cases}
	\\
	&\ge 1 - 2 \cdot (\tfrac12 + \tfrac16\sqrt2)^n
\end{align*}
for every $n \ge 2$ (trivially also for $n=0,1$) and for an arbitrary probability measure $\hat\mu$ on $S \times S$. The classical coupling inequality then provides the upper bound
\begin{equation*}
	\norm{\mu_1 \circ T^n - \mu_2 \circ T^n} \le 4 \cdot (\tfrac12 + \tfrac16 \sqrt2)^n.
\end{equation*}

\begin{rmk}
	A simple computation verifies that $\tfrac12 + \tfrac16 \sqrt2$ is the second largest eigenvalue of the transition matrix $T$. Hence, the asymptotic rate of $(\tfrac12 + \tfrac16\sqrt2)^n$ in the coupling inequality is tight.
\end{rmk}

\pagebreak[3]
\subsection{A Non-Commutative Example}

We now turn to a non-commutative example. The example is motivated by the micromaser experiment in quantum optics, where a quantum harmonic oscillator interacts with a two level atom via the Jaynes-Cummings interactions (see \eg~\cite{Jaynes-Cummings, Meystre-Sargent}. Under some trapping state condition this model has been generalized in \cite{Gohm-etal} and shown to be asymptotically complete. This guarantees that our method can successfully be applied. 

For sake of simplicity we choose model parameters that do not fit the Jaynes-Cummings interaction. Instead we take parameters such that the classical example discussed in Section~\ref{sec:classical_expl} is embedded. A further advantage of our choice is that the computations can still be made by hand. The presented method itself can easily be adopted to any finite-dimensional asymptotically complete system.

We start by introducing some abbreviations. Let $M_3$ denote the algebra of complex $(3\times 3)$-matrices and put 
\begin{align*}
	s &:= \begin{pmatrix}
		0 & 0 & 0 \\ 1 & 0 & 0 \\ 0 & 1 & 0 
	\end{pmatrix},
	&
	a &:= \begin{pmatrix} 
		1 & 0 & 0 \\ 0 & \tfrac12 \sqrt2 & 0 \\ 0 & 0 & \tfrac12 \sqrt2 
	\end{pmatrix},
	&
	a_+ := \begin{pmatrix}
		\tfrac12 \sqrt2 & 0 & 0 \\ 0 & \tfrac12 \sqrt2 & 0 \\ 0 & 0 & 1
	\end{pmatrix}.
\end{align*}
We define a unitary matrix $u \in M_3 \tensor M_2$ as a $2\times2$ block matrix~by
\begin{equation*}
	u := \begin{pmatrix}
		a_+ & i\tfrac12 \sqrt2 \;s^* \\
		i \tfrac12 \sqrt2 \;s & a
	\end{pmatrix}.
\end{equation*}
On the algebra $M_2$ of $(2\times2)$-matrices we fix the state $\psi(y) := \tfrac13 y_{1,1} + \tfrac23 y_{2,2}$ ($y \in M_2$). In this example we study the Markov operator 
\begin{equation*}
	T:M_3 \to M_3,
	\quad
	T(x) := (\Id \tensor \psi) \bigl( u^* (x\tensor \one) u \bigr).
\end{equation*}

\begin{rmk}
	The operator $T$ maps diagonal matrices into diagonal matrices. 
	If we identify diagonal matrices with $\ell^\infty(S)$ and restrict $T$, we recover the preceding example of Section~\ref{sec:classical_expl}.	
\end{rmk}

The operator $T$ already comes along with the dilation given by
\begin{equation*}
	\Gamma:M_3 \to M_3 \tensor M_2, 
	\quad
	\Gamma(x) := u^*(x\tensor \one)x.
\end{equation*}
Moreover, the state on $M_3$ given by $\varphi(x) := \tfrac47 x_{1,1} + \tfrac27 x_{2,2} + \tfrac17 x_{3,3}$ ($x \in M_3$) satisfies \mbox{$(\varphi \tensor \psi) \circ \Gamma = \varphi$} and $\Gamma$ commutes with the modular automorphism groups of $\varphi$ and $\psi$. The hypothesis of Theorem \ref{thm:diag_vs_Z} and \ref{thm:ac_diag} are hence satisfied. 
To study the diagonal coupling associated with $\Gamma$ and $\psi$ we identify the commutants $M_3'$  and $M_2'$ again with $M_3$ and $M_2$, respectively, via the transpose map. Then a Kraus decomposition of the diagonal coupling $\hat T_\Delta$ is given by 
\begin{gather*}
	\hat T_\Delta: M_3 \tensor M_3 \longrightarrow M_3 \tensor M_3
	\\
	\hat T_\Delta(z) = t_1^* z t_1 + t_2^* z t_2 + t_3^* z t_3 + t_4^* z t_4
	\shortintertext{with}
	\begin{aligned}
		t_1 &:= \tfrac13 \sqrt3 \; (a_+ \tensor a_+) + \tfrac16 \sqrt 6 \; (s^* \tensor s^*),
		&
		t_2 &:= \tfrac16 \sqrt6 \; (s \tensor a_+) - \tfrac13 \sqrt 3 \; (a \tensor s^*),
		\\
		t_4 &:= \tfrac16 \sqrt3 \; (s \tensor s) + \tfrac13 \sqrt6 \; (a \tensor a),
		&
		t_3 &:= \tfrac16 \sqrt6 \; (a_+ \tensor s) - \tfrac13 \sqrt 3 \; (s^* \tensor a).
	\end{aligned}
\end{gather*}

\begin{rmk}
	The products $t_i^* t_j$ are linear independent. Hence $\hat T_\Delta$ is an extremal Markov operator; in particular, it is not of the type mentioned in Example~\vref{ex:M_c_convex}.
\end{rmk}

For simplifying the calculations you may observe that each operator $t_i$, $1 \le i \le 4$, respects the decomposition of $\C^3 \tensor \C^3$ into the subspaces
\begin{equation}	\label{eq:subspaces}
	\begin{gathered}
		\H_0 := \lin \{e_1 \tensor e_1, \; e_2 \tensor e_2, \; e_3 \tensor e_3\},
		\\
		\H_{+1} := \lin \{ e_1 \tensor e_2, \; e_2 \tensor e_3 \},
		\qquad
		\H_{-1} := \lin \{ e_2 \tensor e_1, \; e_3 \tensor e_2 \},
		\\
		\H_{+2} := \C e_1 \tensor e_3, 
		\qquad
		\H_{-2} := \C e_3 \tensor e_1,
	\end{gathered}
\end{equation}
where $e_1, e_2, e_3$ denotes the canonical basis of $\C^3$. More precisely, if we denote by $p_k$ the projection onto $\H_k$ for $k \in \{0,\pm1, \pm2\}$ and put $p_k := 0$ otherwise, we have
\begin{align*}
	t_1 p_k &= p_k t_1,
	&
	t_2 p_k &= p_{k-1} t_2,
	&
	t_3 p_k &= p_{k+1} t_3,
	&
	t_4 p_k = p_k t_4
\end{align*}
for every $k \in \Z$. For the support projection $p_\Delta := \supp \varphi_\Delta \le p_0$ a straightforward calculation then yields $p_k \, \hat T_\Delta^2(p_\Delta) \, p_\ell = 0$ for $\ell \neq k$ and
\begin{align*}
	p_2 \, \hat T_\Delta^2(p_\Delta) \, p_2 
	&= t_2^* t_2^* p_\Delta t_2 t_2 
	= \frac{11}{84} - \frac{5}{63} \, \sqrt 2 
	\approx 0.0187 \;,
	\\
	p_1 \, \hat T_\Delta^2(p_\Delta)\,  p_1 
	&= \sum_{i\in \{1,4\}} (t_2^* t_i^* p_\Delta t_i t_2 + t_i^* t_2^* p_\Delta t_2 t_i)
	\\
	&= \frac{1}{1008} \begin{pmatrix}
		379 - 152 \sqrt2 & 54 - 9 \sqrt 2 \\
		54 - 9 \sqrt2 & 190 - 116 \sqrt 2
	\end{pmatrix},
	\\
	p_0 \, \hat T_\Delta^2(p_\Delta)\,  p_0
	&= \sum_{i,j \in \{1,4\}} t_j^* t_i^* p_\Delta t_i t_j + t_3^* t_2^* p_\Delta t_2 t_3 + t_2^* t_3^* p_\Delta t_3 t_2
	\\
	&= \frac{1}{1008} \begin{pmatrix}
		717 - 16 \sqrt2 & 52 + 189 \sqrt2 & 204 - 20 \sqrt2 \\
		52 + 189 \sqrt2 & 483 - 80 \sqrt2 & 56 - 147 \sqrt2 \\
		204 - 20 \sqrt2 & 56 + 147 \sqrt2 & 306 - 16 \sqrt2
	\end{pmatrix},
\end{align*} 
where the matrices on the rightmost side are given with respect to the tensor product basis mentioned in~\eqref{eq:subspaces}. Due to the symmetry under tensor flip the matrices of $p_{-1} \, \hat T_\Delta^2(p_\Delta) \, p_{-1}$ and $p_1 \, \hat T_\Delta^2(p_\Delta) \, p_1$ agree, and likewise for $p_{-2}\, \hat T_\Delta^2(p_\Delta) \, p_{-2}$ and $p_2 \, \hat T_\Delta^2(p_\Delta) \, p_2$. Computing the eigenvalues of these matrices%
\footnote{
	The vector $\xi_\Delta \in \H_0$ corresponding to the pure state $\varphi_\Delta$ always is an eigenvector of $T_\Delta^n(p_\Delta)$. The computations therefore reduce to $(2\times 2)$-matrices.%
}
we find that the smallest eigenvalue occurs for $p_{\pm1} T_\Delta^2(p_\Delta) p_{\pm 1}$ and is given~by 
\begin{equation*}
	r := \frac{1}{2016} \bigl( 569 - 268 \sqrt2 - 9 \sqrt{625 - 216 \sqrt2} \bigr) \approx 0.014.
\end{equation*}
Consequently, we have $T_\Delta^2(p_\Delta) \ge r \one$. Since we also have $T_\Delta^2(p_\Delta) \ge p_\Delta$, it follows $T_\Delta^2(p_\Delta) \ge r \one + (1-r) p_\Delta$. By induction we deduce that
\begin{equation*}
	T_\Delta^{2n}(p_\Delta) \ge \bigl( 1 - (1-r)^n \bigr) \one + (1-r)^n p_\Delta \ge \one - (1-r)^n \one
\end{equation*}
for every $n \in \N$. Now the Quantum Coupling Inequality (Theorem~\ref{thm:QCI}) finally yields
\begin{equation*}
	\norm{\varphi_1 \circ T^n - \varphi_2 \circ T^n} \le 4 (1-r)^{n/4}
\end{equation*}
for every even $n \in \N$.

\begin{rmk}
	The presented method relies on the fact that it follows from asymptotic completeness that $\hat T^n(p_\Delta)$ is strictly positive for some $n \in \N$. In this particular example $n=2$ is the smallest number for which this occurs. Our numeric calculations show that choosing $n > 2$ improves the estimate monotonely. But the limit is not close to the rate given by the second largest eigenvalue of $T$, which is $\tfrac{1}{12} + \tfrac13 \sqrt2 + \tfrac{1}{12} \sqrt5 \approx 0.741$. 
\end{rmk}

\bibliographystyle{alpha}
\bibliography{diagonal_couplings}

\end{document}